\documentclass[12pt,a4paper]{amsart}

\usepackage{amssymb}
\usepackage{amsrefs}
\usepackage[all]{xy}
\usepackage{hyperref} % arXiv

\numberwithin{equation}{section}

\newcounter{claim}[section]

\newtheorem{corollary}[claim]{Corollary}
\newtheorem{lemma}[claim]{Lemma}
\newtheorem{proposition}[claim]{Proposition}
\newtheorem{theorem}[claim]{Theorem}

\newcommand{\calA}{\mathcal{A}}

\newcommand{\calE}{\mathcal{E}}
\newcommand{\calI}{\mathcal{I}}
\newcommand{\calJ}{\mathcal{J}}
\newcommand{\calP}{\mathcal{P}}
\newcommand{\calQ}{\mathcal{Q}}
\newcommand{\calR}{\mathcal{R}}
\newcommand{\calT}{\mathcal{T}}
\newcommand{\calU}{\mathcal{U}}
\newcommand{\calX}{\mathcal{X}}
\newcommand{\calZ}{\mathcal{Z}}

\newcommand{\bbM}{\mathbb{M}}
\newcommand{\bbN}{\mathbb{N}}
\newcommand{\bbP}{\mathbb{P}}
\newcommand{\bbX}{\mathbb{X}}
\newcommand{\bbZ}{\mathbb{Z}}

\newcommand{\bd}{\mathbf{d}}
\newcommand{\be}{\mathbf{e}}
\newcommand{\bh}{\mathbf{h}}
\newcommand{\br}{\mathbf{r}}
\newcommand{\bP}{\mathbf{P}}
\newcommand{\bQ}{\mathbf{Q}}
\newcommand{\bR}{\mathbf{R}}
\newcommand{\bone}{\mathbf{1}}
\newcommand{\bDelta}{\mathbf{\Delta}}

\newcommand{\frakR}{\mathfrak{R}}
\newcommand{\frakX}{\mathfrak{X}}

\let\mod=\undefined
\DeclareMathOperator{\GL}{GL} %
\DeclareMathOperator{\Id}{Id} %
\DeclareMathOperator{\SI}{SI} %
\DeclareMathOperator{\add}{add} %
\DeclareMathOperator{\End}{End} %
\DeclareMathOperator{\Ext}{Ext} %
\DeclareMathOperator{\Hom}{Hom} %
\DeclareMathOperator{\ind}{ind} %
\DeclareMathOperator{\Ker}{Ker} %
\DeclareMathOperator{\mod}{mod} %
\DeclareMathOperator{\rep}{rep} %
\DeclareMathOperator{\chr}{char} %
\DeclareMathOperator{\idim}{idim} %
\DeclareMathOperator{\pdim}{pdim} %
\DeclareMathOperator{\Coker}{Coker} %
\DeclareMathOperator{\gldim}{gldim} %
\DeclareMathOperator{\bdim}{\mathbf{dim}} %

\newcommand{\vertexU}[1]{\bullet \save*+!U{\scriptstyle #1} \restore}

\title[Semi-invariants for concealed-canonical
algebras]{Semi-invariants for concealed-canonical algebras}

\author{Grzegorz Bobi\'nski}

\address{Faculty of Mathematics and Computer Science \\ Nicolaus
Copernicus University \\ ul.~Chopina 12/18 \\ 87-100 Toru\'n \\
Poland}

\email{gregbob@mat.uni.torun.pl}



\begin{document}

\begin{abstract}
In the paper is we generalize known descriptions of rings of
semi-invariants for regular modules over Euclidean and canonical
algebras to arbitrary concealed-canonical algebras.
\end{abstract}

\maketitle

Throughout the paper $\Bbbk$ is a fixed algebraically closed
field. By $\bbZ$, $\bbN$ and $\bbN_+$ we denote the sets of the
integers, the non-negative integers and the positive integers,
respectively. Finally, if $i, j \in \bbZ$, then $[i, j] := \{ k
\in \bbZ \mid i \leq k \leq j \}$ (in particular, $[i, j] =
\varnothing$ if $i > j$).

\section*{Introduction}

Concealed-canonical algebras have been introduced by Lenzing and
Meltzer~\cite{LenzingMeltzer1996} as a generalization of Ringel's
canonical algebras~\cite{Ringel1984}. An algebra is called
concealed-canonical if it is isomorphism to the endomorphism ring
of a tilting bundle over a weighted projective line. The
concealed-canonical algebras can be characterized as the algebras
which posses sincere separating exact
subcategory~\cite{LenzingdelaPena1999} (see
also~\cite{Skowronski1996}). Together with tilted
algebras~\cites{Bongartz1981, HappelRingel1982}, the
concealed-canonical algebras form two most prominent classes of
quasi-tilted algebras~\cite{HappelReitenSmalo1996}. Moreover,
according to a famous result of Happel~\cite{Happel2001}, every
quasi-tilted algebra is derived equivalent either to a tilted
algebra or to a concealed-canonical algebra.

Despite investigations of a structure of the categories of modules
over concealed-canonical algebras, geometric problems have been
studied for this class of algebras (see for
example~\cites{BarotSchroer2001, Bobinski2008,
BobinskiSkowronski2002, GeissSchroer2003, DomokosLenzing2000,
DomokosLenzing2002, SkowronskiWeyman1999}). Often these problem
were studied for canonical algebras only and sometimes the authors
restrict their attention to the concealed-canonical algebras of
tame representation type.

In the paper we study a problem, which has been already
investigated in the case of canonical algebras. Namely, given a
concealed-canonical algebra $\Lambda$ and a module $R$, which is a
direct sum of modules from of sincere separating exact subcategory
of $\mod \Lambda$, we want to describe a structure of the ring of
semi-invariants associated to $\Lambda$ and the dimension vector
of $R$. This problem has been solved provided $\Lambda$ is a
canonical algebra and $R$ comes from a distinguished sincere
separating exact subcategory of $\mod \Lambda$ (the answers have
been obtained independently by Skowro\'nski and
Weyman~\cite{SkowronskiWeyman1999} and Domokos and
Lenzing~\cites{DomokosLenzing2000, DomokosLenzing2002}). This
problem has also been solved for another class of
concealed-canonical algebras, namely the path algebras of
Euclidean quivers~\cite{SkowronskiWeyman2000} (see
also~\cites{DiTrapano2010, Shmelkin1997}). The obtained results
are very similar, although the methods used in the proof are
completely different. The aim my paper is to obtain a unified
proof of the above results, which would generalize to an arbitrary
concealed-canonical algebra. This aim is achieved if the
characteristic of $\Bbbk$ equals $0$. If $\chr \Bbbk > 0$, then we
show that an analogous result is true if we study the
semi-invariants which are the restrictions of the semi-invariants
on the ambient affine space. The precise formulation of the
obtained results can be found in Section~\ref{section result}. In
particular we prove that the studied rings of semi-invariants are
always complete intersections, and are polynomial rings if the
considered dimension vector is ``sufficiently big''.

The paper is organized as follows. In Section~\ref{section
quivers} we introduce a setup of quivers and their
representations, which due to a result of
Gabriel~\cite{Gabriel1972} is an equivalent way of thinking about
algebras and modules. Next, in Section~\ref{section tubular} we
gather facts about concealed-canonical algebras (equivalently,
quivers). In Section~\ref{section semiinvariants} we introduce
semi-invariants and present their basic properties. Next, in
Section~\ref{section preliminary} we study the semi-invariants in
the case of concealed-canonical quivers more closely.
Section~\ref{section Kronecker} is devoted to presentation of
necessary facts about the Kronecker quiver, which is the minimal
concealed-canonical quiver. Finally, in Section~\ref{section
result} we present and proof the main result.

The author gratefully acknowledges the support of the Alexander
von Humboldt Foundation. The research was also supported by
National Science Center Grant No.\ DEC-2011/03/B/ST1/00847.

\section{Quivers and their representations} \label{section
quivers}

By a quiver $\Delta$ we mean a finite set $\Delta_0$ (called the
set of vertices of $\Delta$) together with a finite set $\Delta_1$
(called the set of arrows of $\Delta$) and two maps $s, t :
\Delta_1 \to \Delta_0$, which assign to each arrow $\alpha$ its
starting vertex $s \alpha$ and its terminating vertex $t \alpha$,
respectively. By a path of length $n \in \bbN_+$ in a quiver
$\Delta$ we mean a sequence $\sigma = (\alpha_1, \ldots,
\alpha_n)$ of arrows such that $s \alpha_i = t \alpha_{i + 1}$ for
each $i \in [1, n - 1]$. In the above situation we put $\ell
\sigma := n$, $s \sigma := s \alpha_n$ and $t \sigma := t
\alpha_1$. We treat every arrow in $\Delta$ as a path of length
$1$. Moreover, for each vertex $x$ we have a trivial path
$\bone_x$ at $x$ such that $\ell \bone_x := 0$ and $s \bone_x := x
=: t \bone_x$. For the rest of the paper we assume that the
considered quivers do not have oriented cycles, where by an
oriented cycle we mean a path $\sigma$ of positive length such
that $s \sigma = t \sigma$.

Let $\Delta$ be a quiver. We define its path category $\Bbbk
\Delta$ to be the category whose objects are the vertices of
$\Delta$ and, for $x, y \in \Delta_0$, the morphisms from $x$ to
$y$ are the formal $\Bbbk$-linear combinations of paths starting
at $x$ and terminating at $y$. If $\omega$ is a morphism from $x$
to $y$, then we write $s \omega := x$ and $t \omega := y$. By a
representation of $\Delta$ we mean a functor from $\Bbbk \Delta$
to the category $\mod \Bbbk$ of finite dimensional vector spaces.
We denote the category of representations of $\Delta$ by $\rep
\Delta$. Observe that every representation of $\Delta$ is uniquely
determined by its values on the vertices and the arrows. Given a
representation $M$ of $\Delta$ we denote by $\bdim M$ its
dimension vector defined by the formula $(\bdim M) (x) :=
\dim_\Bbbk M (x)$, for $x \in \Delta_0$. Observe that $\bdim M \in
\bbN^{\Delta_0}$ for each representation $M$ of $\Delta$. We call
the elements of $\bbN^{\Delta_0}$ dimension vectors. A dimension
vector $\bd$ is called sincere if $\bd (x) \neq 0$ for each $x \in
\Delta_0$.

By a relation in a quiver $\Delta$ we mean a $\Bbbk$-linear
combination of paths of lengths at least $2$ having a common
starting vertex and a common terminating vertex. Note that each
relation in a quiver $\Delta$ is a morphism in $\Bbbk \Delta$. A
set $\frakR$ of relations in a quiver $\Delta$ is called minimal
if $\langle \frakR \setminus \{ \rho \} \rangle \neq \langle
\frakR \rangle$ for each $\rho \in \frakR$, where for a set
$\frakX$ of morphisms in $\Delta$ we denote by $\langle \frakX
\rangle$ the ideal in $\Bbbk \Delta$ generated by $\frakX$.
Observe that each minimal set of relations is finite. By a bound
quiver $\bDelta$ we mean a quiver $\Delta$ together with a minimal
set $\frakR$ of relations. Given a bound quiver $\bDelta$ we
denote by $\Bbbk \bDelta$ its path category, i.e.\ $\Bbbk \bDelta
:= \Bbbk \Delta / \langle \frakR \rangle$. By a representation of
a bound quiver $\bDelta$ we mean a functor from $\Bbbk \bDelta$ to
$\mod \Bbbk$. In other words, a representation of $\bDelta$ is a
representation $M$ of $\Delta$ such that $M (\rho) = 0$ for each
$\rho \in \frakR$. We denote the category of representations of a
bound quiver $\bDelta$ by $\rep \bDelta$. Moreover, we denote by
$\ind \bDelta$ the full subcategory of $\rep \bDelta$ consisting
of the indecomposable representations. It is known that $\rep
\bDelta$ is an abelian Krull--Schmidt category.

An important role in the study of representations of quivers is
played by the Auslander--Reiten translations $\tau$ and
$\tau^-$~\cite{AssemSimsonSkowronski2006}*{Section~IV.2}, which
assign to each representation of a bound quiver $\bDelta$ another
representation of $\bDelta$. In particular, we will use the
following consequences of the Auslander--Reiten
formulas~\cite{AssemSimsonSkowronski2006}*{Theorem~IV.2.13}. Let
$M$ and $N$ be representations of a bound quiver $\bDelta$. If
$\pdim_\bDelta M \leq 1$, then
\begin{equation} \label{ARformula1}
\dim_\Bbbk \Ext_\bDelta^1 (M, N) = \dim_\Bbbk \Hom_\bDelta (N,
\tau M).
\end{equation}
Dually, if $\idim_\bDelta N \leq 1$, then
\begin{equation} \label{ARformula2}
\dim_\Bbbk \Ext_\bDelta^1 (M, N) = \dim_\Bbbk \Hom_\bDelta (\tau^-
N, M).
\end{equation}

Let $\bDelta$ be a bound quiver. We define the corresponding Tits
form $\langle -, - \rangle_\bDelta : \bbZ^{\Delta_0} \times
\bbZ^{\Delta_0} \to \bbZ$ by the formula
\[
\langle \bd', \bd'' \rangle_\bDelta := \sum_{x \in \Delta_0} \bd'
(x) \cdot \bd'' (x) - \sum_{\alpha \in \Delta_1} \bd' (s \alpha)
\cdot \bd'' (t \alpha) + \sum_{\rho \in \frakR} \bd' (s \rho)
\cdot \bd'' (t \rho),
\]
for $\bd', \bd'' \in \bbZ^{\Delta_0}$.
Bongartz~\cite{Bongartz1983}*{Proposition~2.2} has proved that
\begin{multline*}
\langle \bdim M, \bdim N \rangle_\bDelta
\\
= \dim_\Bbbk \Hom_\bDelta (M, N) - \dim_\Bbbk \Ext_\bDelta^1 (M,
N) + \dim_\Bbbk \Ext_\bDelta^2 (M, N)
\end{multline*}
for any $M, N \in \rep \bDelta$ provided $\gldim \bDelta \leq 2$.

\section{Separating exact subcategories} \label{section tubular}

In this section we present facts about sincere separating exact
subcategories, which we use in our considerations. For the proofs
we refer to~\cites{Ringel1984, LenzingdelaPena1999}.

Let $\bDelta$ be a bound quiver and $\calX$ a full subcategory of
$\ind \bDelta$. We denote by $\add \calX$ the full subcategory of
$\rep \bDelta$ formed by the direct sums of representations from
$\calX$. We say that $\calX$ is an exact subcategory of $\ind
\bDelta$ if $\add \calX$ is an exact subcategory of $\rep
\bDelta$, where by an exact subcategory of $\rep \bDelta$ we mean
a full subcategory $\calE$ of $\rep \bDelta$ such that $\calE$ is
an abelian category and the inclusion functor $\calE
\hookrightarrow \rep \bDelta$ is exact. We put
\begin{gather*}
\calX_- := \{ X \in \ind \bDelta : \text{$\Hom_\bDelta (\calX, X)
= 0$} \}
\\
\intertext{and} %
\calX_+ := \{ X \in \ind \bDelta : \text{$\Hom_\bDelta (X, \calX)
= 0$} \}.
\end{gather*}

Let $\bDelta$ be a bound quiver.
Following~\cite{LenzingdelaPena1999} we say that $\calR$ is a
sincere separating exact subcategory of $\ind \bDelta$ provided
the following conditions are satisfied:
\begin{enumerate}

\item
$\calR$ is an exact subcategory of $\ind \bDelta$ stable under the
actions of the Auslander--Reiten translations $\tau$ and $\tau^-$.

\item
$\ind \bDelta = \calR_- \cup \calR \cup \calR_+$.

\item
$\Hom_\bDelta (X, \calR) \neq 0$ for each $X \in \calR_-$ and
$\Hom_\bDelta (\calR, X) \neq 0$ for each $X \in \calR_+$.

\item
$P \in \calR_-$, for each indecomposable projective representation
$P$ of $\bDelta$, and $I \in \calR_+$, for each indecomposable
injective representation $I$ of $\bDelta$.

\end{enumerate}
Lenzing and de la Pe\~na~\cite{LenzingdelaPena1999} have proved
that there exists a sincere separating exact subcategory $\calR$
of $\ind \bDelta$ if and only if $\bDelta$ is concealed-canonical,
i.e.\ $\rep \bDelta$ is equivalent to the category of modules over
a concealed-canonical algebra. In particular, if this is the case,
then $\gldim \bDelta \leq 2$.

For the rest of the section we fix a concealed-canonical bound
quiver $\bDelta$ and a sincere separating exact subcategory
$\calR$ of $\ind \bDelta$. Moreover, we put $\calP := \calR_-$ and
$\calQ := \calR_+$. Finally, we denote by $\bP$, $\bR$ and $\bQ$
the dimension vectors of the representations from $\add \calP$,
$\add \calR$ and $\add \calQ$, respectively.

It is known that $\pdim_\bDelta P \leq 1$ for each $P \in \calP$
and $\idim_\bDelta Q \leq 1$ for each $Q \in \calQ$. Next,
$\pdim_\bDelta R = 1$ and $\idim_\bDelta R = 1$ for each $R \in
\calR$. The categories $\calP$ and $\calQ$ are closed under the
actions of $\tau$ and $\tau^-$, hence using the Auslander--Reiten
formulas~\eqref{ARformula1} and~\eqref{ARformula2} we obtain that
$\Ext_\bDelta^1 (\calP, \calR) = 0 = \Ext_\bDelta^1 (\calR,
\calQ)$. In particular,
\begin{equation} \label{eq PQR}
\langle \bd', \bd \rangle_\bDelta \geq 0 \qquad \text{and} \qquad
\langle \bd, \bd'' \rangle_\bDelta \geq 0
\end{equation}
for all $\bd' \in \bP$, $\bd \in \bR$ and $\bd'' \in \bQ$.

We have $\calR = \coprod_{\lambda \in \bbP_\Bbbk^1} \calR_\lambda$
for connected uniserial categories $\calR_\lambda$, $\lambda \in
\bbP_\Bbbk^1$. For $\lambda \in \bbP_\Bbbk^1$ we denote by
$r_\lambda$ the number of the pairwise non-isomorphic simple
objects in $\add \calR_\lambda$. Then $r_\lambda < \infty$ for
each $\lambda \in \bbP_\Bbbk^1$.  Moreover, $\sum_{\lambda \in
\bbP_\Bbbk^1} (r_\lambda - 1) = |\Delta_0| - 2$. In particular, if
$\bbX_0 := \{ \lambda \in \bbP_\Bbbk^1 : \text{$r_\lambda > 1$}
\}$, then $|\bbX_0| < \infty$.

Fix $\lambda \in \bbP_\Bbbk^1$. If $R_{\lambda, 0}$, \ldots,
$R_{\lambda, r_\lambda - 1}$ are chosen representatives of the
isomorphisms classes of the simple objects in $\add
\calR_\lambda$, then we may assume that $\tau R_{\lambda, i} =
R_{\lambda, i - 1}$ for each $i \in [0, r_\lambda - 1]$, where we
put $R_{\lambda, i} := R_{\lambda, i \mod r_\lambda}$, for $i \in
\bbZ$. For any $i \in \bbZ$ and $n \in \bbN_+$ there exists a
unique (up to isomorphism) representation in $\calR_\lambda$ whose
socle and length in $\add \calR_\lambda$ are $R_{\lambda, i}$ and
$n$, respectively. We fix such representation and denote it by
$R_{\lambda, i}^{(n)}$ and its dimension vector by $\be_{\lambda,
i}^n$. Then the composition factors of $R_{\lambda, i}^{(n)}$ are
(starting from the socle) $R_{\lambda, i}$, \ldots, $R_{\lambda, i
+ n - 1}$. Consequently, $\be_{\lambda, i}^n = \sum_{j \in [i, i +
n - 1]} \be_{\lambda, j}$, where $\be_{\lambda, j} := \bdim
R_{\lambda, j}$, for $j \in \bbZ$. Moreover, for all $i \in \bbZ$
and $n, m \in \bbN_+$ there exists an exact sequence
\begin{equation} \label{eq sequence}
0 \to R_{\lambda, i}^{(n)} \to R_{\lambda, i}^{(n + m)} \to
R_{\lambda, i + n}^{(m)} \to 0.
\end{equation}
Obviously, for each $R \in \calR_\lambda$ there exist $i \in \bbZ$
and $n \in \bbN_+$ such that $R \simeq R_{\lambda, i}^{(n)}$.
Moreover, it is known that the vectors $\be_{\lambda, 0}$, \ldots,
$\be_{\lambda, r_\lambda - 1}$ are linearly independent.
Consequently, if $R \in \add \calR_\lambda$, then there exist
uniquely determined $q_0^R, \ldots, q_{r_\lambda - 1}^R \in \bbN$
such that $\bdim R = \sum_{i \in [0, r_\lambda - 1]} q_i^R
\be_{\lambda, i}$. We put $q_i^R := q_{i \mod r_\lambda}^R$, for
$i \in \bbZ$. Observe that for each $i \in \bbZ$ the number
$q_{\lambda, i}^R$ counts the multiplicity of $R_{\lambda, i}$ as
a composition factor in the Jordan--H\"older filtration of $R$ in
the category $\add \calR_\lambda$.

Let $R = \bigoplus_{\lambda \in \bbP_\Bbbk^1} R_\lambda$, for
$R_\lambda \in \add \calR_\lambda$, $\lambda \in \bbP_\Bbbk^1$.
Then we put $q_{\lambda, i}^R := q_i^{R_\lambda}$ for $\lambda \in
\bbP_\Bbbk^1$ and $i \in \bbZ$. Next, we put $p_\lambda^R := \min
\{ q_{\lambda, i}^R : \text{$i \in \bbZ$} \}$, for $\lambda \in
\bbP_\Bbbk^1$, and $p_{\lambda, i}^R := q_{\lambda, i}^R -
p_\lambda^R$, for $\lambda \in \bbP_\Bbbk^1$ and $i \in \bbZ$.
Then
\[
\bdim R = \sum_{\lambda \in \bbP_\Bbbk^1} p_\lambda^R \cdot
\bh_\lambda + \sum_{\lambda \in \bbP_\Bbbk^1} \sum_{i \in [0,
r_\lambda - 1]} p_{\lambda, i}^R \cdot \be_{\lambda, i},
\]
where $\bh_\lambda := \sum_{i \in [0, r_\lambda - 1]}
\be_{\lambda, i}$, for $\lambda \in \bbP_\Bbbk^1$. It is known
that $\bh_\lambda = \bh_\mu$ for any $\lambda, \mu \in
\bbP_\Bbbk^1$. We denote this common value by $\bh$. Then
\[
\bdim R = p^R \cdot \bh + \sum_{\lambda \in \bbP_\Bbbk^1} \sum_{i
\in [0, r_\lambda - 1]} p_{\lambda, i}^R \cdot \be_{\lambda, i},
\]
where $p^R := \sum_{\lambda \in \bbP_\Bbbk^1} p_\lambda^R$. It is
known that if $R, R' \in \add \calR$ and $\bdim R = \bdim R'$,
then $p^R = p^{R'}$ and $p_{\lambda, i}^R = p_{\lambda, i}^{R'}$
for any $\lambda \in \bbP_\Bbbk^1$ and $i \in [0, r_\lambda - 1]$.
Consequently, for each $\bd \in \bR$ there exist uniquely
determined $p^\bd \in \bbN$ and $p_{\lambda, i}^\bd \in \bbN$ for
$\lambda \in \bbP_\Bbbk^1$ and $i \in [0, r_\lambda - 1]$, such
that
\[
\bd = p^\bd \cdot \bh + \sum_{\lambda \in \bbP_\Bbbk^1} \sum_{i
\in [0, r_\lambda - 1]} p_{\lambda, i}^\bd \cdot \be_{\lambda, i}
\]
and for each $\lambda \in \bbP_\Bbbk^1$ there exists $i \in [0,
r_\lambda - 1]$ with $p_{\lambda, i}^\bd = 0$. Again we put
$p_{\lambda, i}^\bd := p_{\lambda, i \mod r_\lambda}^\bd$, for
$\bd \in \bR$, $\lambda \in \bbP_\Bbbk^1$ and $i \in \bbZ$.

It is known that $\bh$ is sincere. Moreover, $\bh$ can be used in
order to distinguish between representations from $\calP$, $\calQ$
and $\calR$. Namely, if $X$ is an indecomposable representation of
$\bDelta$, then
\begin{equation} \label{eq P}
X \in \calP \qquad \text{if and only if} \qquad \langle \bdim X,
\bh \rangle_\bDelta > 0.
\end{equation}
Dually, if $X$ is an indecomposable representation of $\bDelta$,
then
\begin{equation} \label{eq Q}
X \in \calQ \qquad \text{if and only if} \qquad \langle \bh, \bdim
X \rangle_\bDelta > 0.
\end{equation}

Let $\lambda, \mu \in \bbP_\Bbbk^1$, $i, j \in \bbZ$ and $m, n \in
\bbN_+$. Then
\[
\dim_\Bbbk \Hom_\bDelta (R_{\lambda, i}^{(n)}, R_{\mu, j}^{(m)}) =
\min \{ q_{\lambda, i + n - 1}^{R_{\mu, j}^{(m)}}, q_{\mu,
j}^{R_{\lambda, i}^{(n)}} \}
\]
(in particular, $\Hom_\bDelta (R_{\lambda, i}^{(n)}, R_{\mu,
j}^{(m)}) = 0$ if $\lambda \neq \mu$). The above formula, together
with the Auslander--Reiten formula~\eqref{ARformula1}, implies
that
\begin{equation} \label{eq ed}
\langle \be_{\lambda, i}^n, \bd \rangle_\bDelta = p_{\lambda, i +
n - 1}^\bd - p_{\lambda, i - 1}^\bd
\end{equation}
for any $\lambda \in \bbP_\Bbbk^1$, $i \in \bbZ$, $n \in \bbN_+$
and $\bd \in \bR$. In particular,
\begin{equation} \label{eq h}
\langle \bh, \bd \rangle_\bDelta = 0 = \langle \bd, \bh
\rangle_\bDelta
\end{equation}
for each $\bd \in \bR$.

An important role in the proofs will be played by ext-minimal
representations. We call a representation $V$ ext-minimal if there
is no decomposition $V = V_1 \oplus V_2$ with $\Ext_\bDelta^1
(V_1, V_2) \neq 0$. We recall facts on ext-minimal representations
belonging to $\add \calR$.

First assume that $\bd \in \bR$ and $p^\bd = 0$. In this case
there is a unique (up to isomorphism) ext-minimal representation
$W \in \add \calR$ with dimension vector $\bd$, which is
constructed inductively in the following way. For $\lambda \in
\bbP_\Bbbk^1$, let $I_\lambda := \{ i \in [0, r_\lambda - 1] :
\text{$p_{\lambda, i}^\bd \neq 0$ and $p_{\lambda, i - 1}^\bd =
0$} \}$. For $\lambda \in \bbP_\Bbbk^1$ and $i \in I_\lambda$, we
denote by $m_{\lambda, i}$ the minimal $m \in \bbN_+$ such that
$p_{\lambda, i + m}^\bd = 0$. By induction there exists (unique up
to isomorphism) ext-minimal representation $W' \in \add \calR$
with dimension vector $\bd - \sum_{\lambda \in \bbP_\Bbbk^1}
\sum_{i \in I_\lambda} \be_{\lambda, i}^{m_{\lambda, i}}$. Then $W
:= W' \oplus \bigoplus_{\lambda \in \bbP_\Bbbk^1} \bigoplus_{i \in
I_\lambda} R_{\lambda, i}^{(m_{\lambda, i})}$ is ext-minimal.

We will use the following property of the above representation.

\begin{lemma} \label{lemma extminimal}
Assume $\bd \in \bR$ and $p^\bd = 0$. Let $W \in \add \calR$ be an
ext-minimal representation with dimension vector $\bd$. If
$\lambda \in \bbP_\Bbbk^1$, $i \in \bbZ$, $n \in \bbN_+$,
$p_{\lambda, i}^\bd = p_{\lambda, i + n}^\bd$ and $p_{\lambda,
j}^\bd \geq p_{\lambda, i}^\bd$ for each $j \in [i, i + n]$, then
$\Hom_\bDelta (R_{\lambda, i + 1}^{(n)}, W) = 0$.
\end{lemma}

\begin{proof}
Observe that $\Hom_\bDelta (R_{\lambda, i + 1}^{(n)}, R_{\lambda,
k}^{(m_{\lambda, k})}) = 0$ for each $k \in I_\lambda$, since one
easily checks that either $q_{\lambda, i + n}^{R_{\lambda,
k}^{(m_{\lambda, k})}} = 0$ (if $p_{\lambda, i}^\bd = 0$) or
$q_{\lambda, k}^{R_{\lambda, i + 1}^{(n)}} = 0$ (if $p_{\lambda,
i}^\bd > 0$). Now the claim follows by induction.
\end{proof}

Now let $\bd \in \bR$ be arbitrary. The description of the
ext-minimal representations with dimension vector $\bd$, which
belong to $\add \calR$, has been given
in~\cite{Ringel1980}*{Theorem~3.5} (this theorem has been
formulated in the case $\bDelta = (\Delta, \varnothing)$ for a
Euclidean quiver $\Delta$, but its proof translates to an
arbitrary concealed-canonical bound quiver). We will not repeat
the formulation here, but only mention some consequences. First,
if $W \in \add \calR$ and $\bdim W = \bd$, then $W$ is ext-minimal
if and only if $\dim_\Bbbk \End_\bDelta (W) = p^\bd + \langle \bd,
\bd \rangle_\bDelta$. In particular,
\begin{multline} \label{eq minimum}
p^\bd + \langle \bd, \bd \rangle_\bDelta
\\
= \min \{ \dim_\Bbbk \End_\bDelta (W) : \text{$W \in \add \calR$
such that $\bdim W = \bd$} \}
\end{multline}
(here we use also~\cite{Ringel1980}*{Lemma~2.1}). Next, if $W \in
\add \calR$ is an ext-minimal representation with dimension vector
$\bd$ and $W' \in \add \calR$ is an ext-minimal representation
with dimension vector $\bd - p^\bd \cdot \bh$, then there exists
an exact sequence $0 \to \bigoplus_{\lambda \in \bbP_\Bbbk^1}
R_\lambda \to W \to W' \to 0$ with $R_\lambda \in \calR_\lambda$
(in particular, indecomposable) for each $\lambda \in
\bbP_\Bbbk^1$ (obviously, $\bdim R_\lambda$ is a multiplicity of
$\bh$ for each $\lambda \in \bbP_\Bbbk^1$).

\section{Semi-invariants} \label{section semiinvariants}

Let $\bDelta$ be a bound quiver and $\bd$ a dimension vector. By
$\rep_\bDelta (\bd)$ we denote the set of the representations $M$
of $\bDelta$ such that $M (x) = \Bbbk^{\bd (x)}$ for each $x \in
\Delta_0$. We may identify $\rep_\bDelta (\bd)$ with a
Zariski-closed subset of the affine space $\rep_\Delta (\bd) :=
\prod_{\alpha \in \Delta_1} \bbM_{\bd (t \alpha) \times \bd (s
\alpha)} (\Bbbk)$, hence it has a structure of an affine variety.
The group $\GL (\bd) := \prod_{x \in \Delta_0} \GL (\bd (x))$ acts
on $\rep_\Delta (\bd)$ by conjugation: $(g \ast M) (\alpha) := g
(t \alpha) \cdot M (\alpha) \cdot g (s \alpha)^{-1}$, for $g \in
\GL (\bd)$, $M \in \rep_\Delta (\bd)$ and $\alpha \in \Delta_1$.
The set $\rep_\bDelta (\bd)$ is a $\GL (\bd)$-invariant subset of
$\rep_\Delta (\bd)$ and the $\GL (\bd)$-orbits in $\rep_\bDelta
(\bd)$ correspond to the isomorphism classes of the
representations of $\bDelta$ with dimension vector $\bd$. If
$\calX$ is a full subcategory of $\ind \bDelta$, then we denote by
$\calX (\bd)$ the set of $V \in \rep_\bDelta (\bd)$ such that $V
\in \add \calX$.

Let $\Delta$ be a quiver and $\theta \in \bbZ^{\Delta_0}$. We
treat $\theta$ as a $\bbZ$-linear function $\bbZ^{\Delta_0} \to
\bbZ$ in a usual way. If $\bd$ is a dimension vector, then by a
semi-invariant of weight $\theta$ we mean every function $f \in
\Bbbk [\rep_\Delta (\bd)]$ such that $f (g^{-1} \ast M) =
\chi^\theta (g) \cdot f (M)$ for any $g \in \GL (\bd)$ and $M \in
\rep_\Delta (\bd)$, where $\chi^\theta (g) := \prod_{x \in
\Delta_0} (\det g (x))^{\theta (x)}$ for $g \in \GL (\bd)$.

Now let $\bDelta$ be a bound quiver and $\bd$ a dimension vector.
If $\theta \in \bbZ^{\Delta_0}$, then a function $f \in \Bbbk
[\rep_\bDelta (\bd)]$ is called a semi-invariant of weight
$\theta$ if $f$ is the restriction of a semi-invariant of weight
$\theta$ from $\Bbbk [\rep_\Delta (\bd)]$. This definition differs
from the definition used in other papers on the subject (see for
example~\cites{BobinskiRiedtmannSkowronski2008, DerksenWeyman2002,
Domokos2002, DomokosLenzing2002}), however these are the
semi-invariants which one needs to understand in order to study
King's moduli spaces for representations of bound
quivers~\cite{King1994}. Moreover, the two definitions coincide if
the characteristic of $\Bbbk$ equals $0$. We denote the space of
the semi-invariants of weight $\theta$ by $\SI [\bDelta,
\bd]_\theta$. If $\bd$ is sincere, then we put $\SI [\bDelta, \bd]
:= \bigoplus_{\theta \in \bbZ^{\Delta_0}} \SI [\bDelta,
\bd]_\theta$ and call it the algebra of semi-invariants for
$\bDelta$ and $\bd$ (we assume sincerity of $\bd$, since under
this assumption $\bbZ^{\bDelta_0}$ is isomorphic with the
character group of $\GL (\bd)$).

We recall a construction from~\cite{Domokos2002}. Let $\bDelta$ be
a bound quiver. Fix a representation $V$ of $\bDelta$ and define
$\theta^V : \bbZ^{\Delta_0} \to \bbZ$ by the condition:
\[
\theta^V (\bdim M) = \dim_\Bbbk \Hom_\bDelta (V, M) - \dim_\Bbbk
\Hom_\bDelta (M, \tau V)
\]
for each representation $M$ of $\bDelta$. The
formula~\eqref{ARformula1} implies that $\theta^V = \langle \bdim
V, - \rangle_\bDelta$ if $\pdim_\bDelta V \leq 1$. Dually, if $V$
has no indecomposable projective direct summands (i.e.\ $\tau^-
\tau V \simeq
V$~\cite{AssemSimsonSkowronski2006}*{Theorem~IV.2.10}) and
$\idim_\bDelta \tau V \leq 1$, then $\theta^V = - \langle -, \bdim
\tau V \rangle_\bDelta$ by the formula~\eqref{ARformula2}.

Now let $\bd$ be a dimension vector. If $\theta^V (\bd) = 0$, then
we define a function $c_\bd^V \in \Bbbk [\rep_\bDelta (\bd)]$ in
the following way. Let $P_1 \xrightarrow{f} P_0 \to V \to 0$ be
the minimal projective presentation of $V$. One shows that
\begin{gather*}
\dim_\Bbbk \Ker \Hom_\bDelta (f, M) = \dim_\Bbbk \Hom_\bDelta (V,
M)
\\
\intertext{and} %
\dim_\Bbbk \Coker \Hom_\bDelta (f, M) = \dim_\Bbbk \Hom_\bDelta
(M, \tau V),
\end{gather*}
hence
\begin{multline} \label{eq thetad}
\dim_\Bbbk \Hom_\bDelta (P_0, M) - \dim_\Bbbk \Hom_\bDelta (P_1,
M)
\\
= \dim_\Bbbk \Hom_\bDelta (V, M) - \dim_\Bbbk \Hom_\bDelta (M,
\tau V) = \theta^V (\bd) = 0,
\end{multline}
for each $M \in \rep_\bDelta (\bd)$. Thus, we may define $c_\bd^V
\in \Bbbk [\rep_\bDelta (\bd)]$ by the formula $c_\bd^V (M) :=
\det \Hom_\bDelta (f, M)$ for $M \in \rep_\bDelta (\bd)$. Note
that $c_\bd^V$ is defined only up to a non-zero scalar. If $M \in
\rep_\bDelta (\bd)$, then $c_\bd^V (M) = 0$ if and only if
$\Hom_\bDelta (V, M) \neq 0$. Moreover, if $\pdim_\bDelta V \leq
1$ and $M \in \rep_\bDelta (\bd)$, then $c_\bd^V (M) = 0$ if and
only if $\Ext_\bDelta^1 (V, M) \neq 0$. It is known that $c_\bd^V
\in \SI [\bDelta, \bd]_{\theta^V}$. This function depends on the
choice of $f$, but the functions obtained for different $f$'s
differ only by non-zero scalars.

In fact, we could start with an arbitrary $\bd$-admissible
projective presentation, where, for a representation $V$ of a
bound quiver $\bDelta$ and a dimension vector $\bd$, we call a
projective representation $P_1' \to P_0' \to V \to 0$ of $V$
$\bd$-admissible if $\dim_\Bbbk \Hom_\bDelta (P_0', M) =
\dim_\Bbbk \Hom_\bDelta (P_1', M)$ for any (equivalently, some) $M
\in \rep_\bDelta (\bd)$.

\begin{lemma} \label{lemma presentation}
Let $\bDelta$ be a bound quiver, $\bd$ a dimension vector and
$P_1' \xrightarrow{f'} P_0' \to V \to 0$ a $\bd$-admissible
projective presentation of a representation $V$ of $\bDelta$.
\begin{enumerate}

\item
If $\theta^V (\bd) = 0$, then there exists $\xi \in \Bbbk$ such
$\xi \neq 0$ and $c_\bd^V (M) = \xi \cdot \det \Hom_\bDelta (f',
M)$ for each $M \in \rep_\bDelta (\bd)$.

\item
If there exists $M \in \rep_\bDelta (\bd)$ such that $\det
\Hom_\bDelta (f', M) \neq 0$, then $\theta^V (\bd) = 0$.

\end{enumerate}
\end{lemma}

\begin{proof}
Let $P_1 \xrightarrow{f} P_0 \to V \to 0$ be the minimal
projective presentation of $V$. There exists projective
representations $P$ and $Q$ of $\bDelta$ and isomorphisms $g_1 :
P_1' \to P_1 \oplus P \oplus Q$ and $g_0 : P_0' \to P_0 \oplus P$
such that
\[
f' = g_0^{-1} \circ
\begin{bmatrix}
f & 0 & 0
\\
0 & \Id_P & 0
\end{bmatrix}
\circ g_1.
\]
Consequently,
\begin{multline} \label{eq HomfM}
\Hom_\bDelta (f', M) = \Hom_\bDelta (g_1, M)
\\
\circ
\begin{bmatrix}
\Hom_\bDelta (f, M) & 0
\\
0 & \Hom_\bDelta (\Id_P, M)
\\
0 & 0
\end{bmatrix}
\circ \Hom_\bDelta (g_0^{-1}, M)
\end{multline}
for each $M \in \rep_\bDelta (\bd)$. Since the presentation $P_1'
\xrightarrow{f'} P_0' \to V \to 0$ is $\bd$-admissible, \eqref{eq
thetad} implies that the condition $\theta^V (\bd) = 0$ is
equivalent to the condition $\dim_\Bbbk \Hom_\bDelta (Q, M) = 0$
for each $M \in \rep_\bDelta (\bd)$. Together with~\eqref{eq
HomfM} this implies our claims.
\end{proof}

As an immediate consequence we obtain the following.

\begin{corollary} \label{corollary multsemi}
Let $\bDelta$ be a bound quiver, $\bd$ a dimension vector and $0
\to V_1 \to V \to V_2 \to 0$ an exact sequence such that
$\theta^{V_1} (\bd) = 0 = \theta^{V_2} (\bd)$.
\begin{enumerate}

\item \label{point multsemi1}
If $\theta^V (\bd) = 0$, then \textup{(}up to a non-zero
scalar\textup{)} $c_\bd^V = c_\bd^{V_1} \cdot c_\bd^{V_2}$.

\item \label{point multsemi2}
If $c_\bd^{V_1} \cdot c_\bd^{V_2} \neq 0$, then $\theta^V (\bd) =
0$ and \textup{(}up to a non-zero scalar\textup{)} $c_\bd^V =
c_\bd^{V_1} \cdot c_\bd^{V_2}$.

\end{enumerate}
\end{corollary}

\begin{proof}
Let $P_1' \xrightarrow{f'} P_0' \to V_1 \to 0$ and $P_1''
\xrightarrow{f''} P_0' \to V_2 \to 0$ be the minimal projective
presentations of $V_1$ and $V_2$, respectively. Then there exists
a projective presentation of $V$ of the form
\[
P_1' \oplus P_1'' \xrightarrow{f} P_0' \oplus P_0'' \to V \to 0,
\]
where $f = \left[
\begin{smallmatrix}
f' & g \\ 0 & f''
\end{smallmatrix}
\right]$ for some $g \in \Hom_\bDelta (P_1'', P_0')$. One easily
sees that $\det \Hom_\bDelta (f, M) = c_\bd^{V_1} (M) \cdot
c_\bd^{V_2} (M)$ for each $M \in \rep_\bDelta (\bd)$, hence the
claims follows from Lemma~\ref{lemma presentation}.
\end{proof}

The following fact is an extension
of~\cite{DerksenWeyman2000}*{Lemma~1(a)} to the setup of bound
quivers.

\begin{lemma} \label{lemma zero}
Let $\bDelta$ be a bound quiver and $\bd$ a dimension vector. If
$0 \to V_1 \to V \to V_2 \to 0$ is an exact sequence, $\theta^V
(\bd) = 0$ and $c_\bd^V \neq 0$, then $\theta^{V_2} (\bd) \leq 0$.
\end{lemma}

\begin{proof}
If $\theta^{V_2} (\bd) > 0$, then
\[
\dim_\Bbbk \Hom_\bDelta (V_2, M) \geq \theta^{V_2} (\bd) > 0
\]
for each $M \in \rep_\bDelta (\bd)$. This immediately implies that
$\Hom_\bDelta (V, M) \neq 0$ for each $M \in \rep_\bDelta (\bd)$,
hence $c_\bd^V = 0$, contradiction.
\end{proof}

We have the following multiplicative property.

\begin{lemma} \label{lemma directsum}
Let $\bDelta$ be a bound quiver and $\bd$ a dimension vector. If
$V_1$ and $V_2$ are representations of $\bDelta$, $V := V_1 \oplus
V_2$, $\theta^V (\bd) = 0$ and $c_\bd^V \neq 0$, then
$\theta^{V_1} (\bd) = 0 = \theta^{V_2} (\bd)$ and $c_\bd^V =
c_\bd^{V_1} \cdot c_\bd^{V_2}$ \textup{(}up to a non-zero
scalar\textup{)}.
\end{lemma}

\begin{proof}
See~\cite{Domokos2002}*{Lemma~3.3}.
\end{proof}

We will also use another multiplicative property.

\begin{lemma} \label{lemma directsumbis}
Let $\bDelta$ be a bound quiver and $V$ a representation of
$\bDelta$. If $\bd'$ and $\bd''$ are dimension vectors and
$\theta^V (\bd') = 0 = \theta^V (\bd'')$, then
\[
c_{\bd' + \bd''}^V (W' \oplus W'') = c_{\bd'}^V (W') \cdot
c_{\bd''}^V (W'')
\]
for all $(W', W'') \in \rep_\bDelta (\bd') \times \rep_\bDelta
(\bd'')$.
\end{lemma}

\begin{proof}
Let $P_1 \xrightarrow{f} P_0 \to V \to 0$ be the minimal
projective presentation of $V$. If $(W', W'') \in \rep_\bDelta
(\bd') \times \rep_\bDelta (\bd'')$, then
\[
\Hom_\bDelta (f, W' \oplus W'') =
\begin{bmatrix}
\Hom_\bDelta (f, W') & 0
\\
0 & \Hom_\bDelta (f, W'')
\end{bmatrix}
\]
and both $\Hom_\bDelta (f, W')$ and $\Hom_\bDelta (f, W'')$ are
square matrices, hence the claim follows.
\end{proof}

The following result follows from the proof
of~\cite{Domokos2002}*{Theorem~3.2} (note that the assumption
about the characteristic of $\Bbbk$ made
in~\cite{Domokos2002}*{Theorem~3.2} is only necessary to prove
surjectivity of the restriction morphism, which we have for free
with our definition of semi-invariants).

\begin{proposition} \label{proposition Domokos}
Let $\bDelta$ be a bound quiver, $\bd$ a dimension vector and
$\theta \in \bbZ^{\Delta_0}$.
\begin{enumerate}

\item
If $\theta (\bd) \neq 0$, then $\SI [\bDelta, \bd]_\theta =0$.

\item
If $\theta (\bd) = 0$, then the space $\SI [\bDelta, \bd]_\theta$
is spanned by the functions $c_\bd^V$ for $V \in \rep \bDelta$
such that $\theta^V = \theta$ and $c_\bd^V \neq 0$. \qed

\end{enumerate}
\end{proposition}

In fact we may take a smaller spanning set.

\begin{corollary} \label{corollary spanning}
Let $\bDelta$ be a bound quiver and $\bd$ a dimension vector. If
$\theta \in \bbZ^{\Delta_0}$ and $\theta (\bd) = 0$, then the
space $\SI [\bDelta, \bd]_{\theta}$ is spanned by the functions
$c_\bd^V$ for ext-minimal $V \in \rep \bDelta$ such that $\theta^V
= \theta$ and $c_\bd^V \neq 0$.
\end{corollary}

\begin{proof}
Assume that $V$ is a representation of $\bDelta$ such that
$\theta^V = \theta$, $c_\bd^V \neq 0$ and there is a decomposition
$V = V_1 \oplus V_2$ with $\Ext_\bDelta^1 (V_1, V_2) \neq 0$.
Lemma~\ref{lemma directsum} implies that $\theta^{V_1} (\bd) = 0 =
\theta^{V_2} (\bd)$ and $c_\bd^{V_1} \cdot c_\bd^{V_2} \neq 0$. If
$0 \to V_2 \to W \to V_1 \to 0$ is a non-split exact sequence,
then Corollary~\ref{corollary multsemi}\eqref{point multsemi2} and
Lemma~\ref{lemma directsum} imply that (up to a non-zero scalar)
$c_\bd^W = c_\bd^{V_1} \cdot c_\bd^{V_2} = c_\bd^V$. Since
$\dim_\Bbbk \End_\bDelta (W) < \dim_\Bbbk \End_\bDelta (V)$ (see
for example~\cite{Ringel1980}*{Lemma~2.1}), the claim follows by
induction.
\end{proof}

We may even take a smaller set, if we are only interested in
generators of $\SI [\bDelta, \bd]$. Namely, we have the following.

\begin{corollary} \label{corollary generating}
Let $\bDelta$ be a bound quiver and $\bd$ a sincere dimension
vector. Then the algebra $\SI [\bDelta, \bd]$ is generated by the
semi-invariants $c_\bd^V$ for $V \in \rep_\bDelta (\bd)$ such that
$\theta^V (\bd) = 0$, $c_\bd^V \neq 0$ and $V$ is indecomposable.
\end{corollary}

\begin{proof}
This follows from Proposition~\ref{proposition Domokos} and
Lemma~\ref{lemma directsum} (this is also the content
of~\cite{Domokos2002}*{Corollary~3.4}).
\end{proof}

\section{Preliminary results} \label{section preliminary}

Throughout this section we fix a concealed-canonical bound quiver
$\bDelta$ and a sincere separating exact subcategory $\calR$ of
$\ind \bDelta$. We will use notation introduced in
Section~\ref{section tubular}. We also fix $\bd \in \bR$ such that
$p := p^\bd > 0$. Notice that this implies that $\bd$ is sincere.

First we prove that the algebra $\SI [\bDelta, \bd]$ is controlled
by the representations from $\add \calR$.

\begin{lemma} \label{lemma R}
Let $V$ be a representation of $\bDelta$ such that $\theta^V (\bd)
= 0$. If $c_\bd^V \neq 0$, then $V \in \add \calR$ and $\theta^V =
\langle \bdim V, - \rangle_\bDelta$.
\end{lemma}

\begin{proof}
Assume that $P \in \calP$ is a direct summand of $V$. Since
$\pdim_\bDelta P \leq 1$, \eqref{eq PQR} and~\eqref{eq P} imply
that
\[
\theta^P (\bd) = \langle \bdim P, \bd \rangle_\bDelta \geq \langle
\bdim P, \bh \rangle_\bDelta > 0.
\]
Consequently, $c_\bd^V = 0$ by Lemma~\ref{lemma directsum},
contradiction. Dually, $V$ cannot have a direct summand from
$\calQ$. Finally, since $\pdim_\bDelta V = 1$, $\theta^V = \langle
\bdim V, - \rangle_\bDelta$.
\end{proof}

Together with Corollary~\ref{corollary spanning} this lemma
immediately implies the following.

\begin{corollary} \label{corollary spanningbis}
Let $\theta \in \bbZ^{\bDelta_0}$ be such that $\SI [\bDelta,
\bd]_\theta \neq 0$. Then there exists $\br \in \bR$ such that
$\theta = \langle \br, - \rangle_\bDelta$ and $\langle \br, \bd
\rangle_\bDelta = 0$. \qed
\end{corollary}

Taking into account Corollary~\ref{corollary generating} we need
to identify $V \in \ind \bDelta$ such that $\theta^V (\bd) = 0$
and $c_\bd^V \neq 0$. The first step in this direction is the
following.

\begin{lemma} \label{lemma indecomposable}
Let $V$ be an indecomposable representation of $\bDelta$. If
$\theta^V (\bd) = 0$ and $c_\bd^V \neq 0$, then $V = R_{\lambda, i
+ 1}^{(n)}$ for some $\lambda \in \bbP_\Bbbk^1$, $i \in \bbZ$ and
$n \in \bbN_+$ such that $p_{\lambda, i}^\bd = p_{\lambda, i +
n}^\bd$ and $p_{\lambda, j}^\bd \geq p_{\lambda, i}^\bd$ for each
$j \in [i + 1, i + n - 1]$.
\end{lemma}

\begin{proof}
We know from Lemma~\ref{lemma R} that $V \in \calR$, hence there
exists $\lambda \in \bbP_\Bbbk^1$, $i \in \bbZ$ and $n \in \bbN_+$
such that $V = R_{\lambda, i + 1}^{(n)}$. Then $\theta^V (\bd) =
p_{\lambda, i + n}^\bd - p_{\lambda, i}^\bd$ by~\eqref{eq ed},
thus the condition $\theta^V (\bd) = 0$ means that $p_{\lambda,
i}^\bd = p_{\lambda, i + n}^\bd$. Finally, the condition $c_\bd^V
\neq 0$ and Lemma~\ref{lemma zero} imply that $\theta^{V'} (\bd)
\leq 0$ for each factor representation $V'$ of $V$. The
sequence~\eqref{eq sequence} implies that $R_{\lambda, j + 1}^{(n
+ i - j)}$ is a factor representation of $V$ for each $j \in [i +
1, i + n - 1]$, hence the claim follows.
\end{proof}

Now we show that the representations described in the above lemma
give rise to non-zero semi-invariants.

\begin{lemma} \label{lemma non-zero1}
Let $\lambda \in \bbP_\Bbbk^1$, $i \in \bbZ$ and $n \in \bbN$ be
such that $p_{\lambda, i}^\bd = p_{\lambda, i + n}^\bd$ and
$p_{\lambda, j}^\bd \geq p_{\lambda, i + n}^\bd$ for each $j \in
[i + 1, i + n - 1]$. If $V := R_{\lambda, i + 1}^{(n)}$, then
$\theta^V (\bd) = 0$ and there exists $R \in \calR (\bd)$ such
that $c_\bd^V (R) \neq 0$.
\end{lemma}

\begin{proof}
We only need to show that there exists $R \in \calR (\bd)$ such
that $c_\bd^V (R) \neq 0$. Let $W \in \add \calR$ be an
ext-minimal representation for $\bd - p \cdot \bh$ and fix $\mu
\in \bbP_\Bbbk^1$ different from $\lambda$ such that $r_\mu = 1$.
If $R := W \oplus R_{\mu, 0}^{(p)}$, then $R \in \rep_\bDelta
(\bd)$ and $\Hom_\bDelta (V, R) = \Hom_\bDelta (V, W) = 0$ by
Lemma~\ref{lemma extminimal}, hence the claim follows.
\end{proof}

As a consequence we present a smaller generating set of $\SI
[\bDelta, \bd]$. First we introduce some notation. For $\lambda
\in \bbP_\Bbbk^1$ we denote by $\calI_\lambda$ the set of $i \in
[0, r_\lambda - 1]$ such that there exists $n \in \bbN_+$ with
$p_{\lambda, i}^\bd = p_{\lambda, i + n}^\bd$ and $p_{\lambda,
j}^\bd > p_{\lambda, i}^\bd$ for each $j \in [i + 1, i + n - 1]$
(such $n$, if exists, is uniquely determined by $\lambda$ and $i$,
and we denote it by $n_{\lambda, i}$). Observe that $\calI_\lambda
= \{ 0 \}$ and $n_{\lambda, 0} = 1$ if $r_\lambda = 1$.

\begin{corollary} \label{corollary generatingbis}
The algebra $\SI [\bDelta, \bd]$ is generated by the
semi-inva\-riants $c_\bd^{R_{\lambda, i + 1}^{(n_{\lambda, i})}}$
for $\lambda \in \bbP_\Bbbk^1$ and $i \in \calI_\lambda$.
\end{corollary}

\begin{proof}
For $\lambda \in \bbP_\Bbbk^1$ we denote by $\calI_\lambda'$ the
set of all pairs $(i, n) \in [0, r_\lambda - 1] \times \bbN_+$
such that $p_{\lambda, i}^\bd = p_{\lambda, i + n}^\bd$ and
$p_{\lambda, j}^\bd \geq p_{\lambda, i}^\bd$ for each $j \in [i +
1, i + n - 1]$. Observe that if $\lambda \in \bbP_\Bbbk^1$ and
$(i, n) \in \calI_\lambda'$, then $i \in \calI_\lambda$.
Corollary~\ref{corollary generating} and Lemma~\ref{lemma
indecomposable} imply that the algebra $\SI [\bDelta, \bd]$ is
generated by the semi-invariants $c_\bd^{R_{\lambda, i +
1}^{(n)}}$ for $\lambda \in \bbP_\Bbbk^1$ and $(i, n) \in
\calI_\lambda'$. Now, let $\lambda \in \bbP_\Bbbk^1$ and $(i, n)
\in \calI_\lambda'$. Obviously, $n \geq n_{\lambda, i_\lambda}$.
If $n > n_{\lambda, i_\lambda}$, then (up to a non-zero scalar)
$c_\bd^{R_{\lambda, i + 1}^{(n)}} = c_\bd^{R_{\lambda, i +
1}^{(n_{\lambda, i})}} \cdot c_\bd^{R_{\lambda, i + n_{\lambda, i}
+ 1}^{(n - n_{\lambda, i})}}$ by Corollary~\ref{corollary
multsemi}\eqref{point multsemi1}, as according to~\eqref{eq
sequence} we have an exact sequence
\[
0 \to R_{\lambda, i + 1}^{(n_{\lambda, i})} \to R_{\lambda, i +
1}^{(n)} \to R_{\lambda, i + n_{\lambda, i} + 1}^{(n - n_{\lambda,
i})} \to 0.
\]
Since $R_{\lambda, i + n_{\lambda, i} + 1}^{(n - n_{\lambda, i})}
= R_{\lambda, (i + n_{\lambda, i} + 1) \mod r_\lambda}^{(n -
n_{\lambda, i})}$ and $((i + n_{\lambda, i}) \mod r_\lambda, n -
n_{\lambda, i}) \in \calI_\lambda'$, the claim follows by
induction.
\end{proof}

At the later stage we will prove that for each non-zero
semi-invariant $f$ there exists $R \in \calR (\bd)$ such that $f
(R) \neq 0$. At the moment we formulate the following versions of
this fact.

\begin{lemma} \label{lemma non-zero3}
Let $V$ be a representation of $\bDelta$ such that $\theta^V (\bd)
= 0$ and $c_\bd^V \neq 0$. Then there exists $R \in \calR (\bd)$
such that $c_\bd^V (R) \neq 0$.
\end{lemma}

\begin{proof}
Let $X$ be an indecomposable direct summand of $V$.
Lemma~\ref{lemma directsum} implies that $c_\bd^X \neq 0$.
Consequently, Lemmas~\ref{lemma indecomposable} and~\ref{lemma
non-zero1} imply that there exists $R_X \in \calR (\bd)$ such that
$c_\bd^X (R_X) \neq 0$. Since $\calR (\bd)$ is an irreducible and
open subset of $\rep_\bDelta
(\bd)$~\cite{DomokosLenzing2002}*{Section~4}, there exists $R \in
\calR (\bd)$ such that $c_\bd^X (R) \neq 0$ for each
indecomposable direct summand $X$ of $V$. Using once more
Lemma~\ref{lemma directsum} we obtain that $c_\bd^V (R) \neq 0$.
\end{proof}

\begin{lemma} \label{lemma non-zero2}
If $q \in \bbN$ and $f \in \SI [\bDelta, \bd]_{\langle q \cdot
\bh, - \rangle_\bDelta}$ is non-zero, then there exists $R \in
\calR (\bd)$ such that $f (R) \neq 0$.
\end{lemma}

\begin{proof}
If $q = 0$, then the claim is obvious, since $\SI [\bDelta, \bd]_0
= \Bbbk$. Thus assume $q > 0$. We know that $\SI [\bDelta,
\bd]_{\langle q \cdot \bh, - \rangle_\bDelta}$ is spanned by the
functions $c_\bd^V$ for $V \in \add \calR$ with dimension vector
$q \cdot \bh$. It is enough to prove that $c_\bd^V (M) = 0$ for
all $V \in \add \calR$ and $M \in \rep_\bDelta (\bd)$ such that
$\bdim V = q \cdot \bh$ and $M \not \in \calR (\bd)$. Every such
$M$ has an indecomposable direct summand $Q$ from $\calQ$. Indeed,
since $M \not \in \calR (\bd)$, it has an indecomposable direct
summand $X$ which belongs to $\calP \cup \calQ$. If $X \in \calQ$,
then we take $Q := X$. If $X \in \calP$, then $\langle \bdim M -
\bdim X, \bh \rangle_\bDelta < 0$ by~\eqref{eq P} and~\eqref{eq
h}. Consequently, $M$ has an indecomposable direct summand $Q$
with $\langle \bdim Q, \bh \rangle_\bDelta < 0$. Using
again~~\eqref{eq P} and~\eqref{eq h} we get $Q \in \calQ$. Then
\[
\dim_\Bbbk \Hom_\bDelta (V, M) \geq \dim_\Bbbk \Hom_\bDelta (V, Q)
= \langle q \cdot \bh, \bdim Q \rangle_\bDelta > 0
\]
by~\eqref{eq Q} and the claim follows.
\end{proof}

Recall from Corollary~\ref{corollary spanningbis} that the
possible weights are of the form $\langle \br, - \rangle_\bDelta$
for $\br \in \bR$ such that $\langle \br, \bd \rangle_\bDelta =
0$. Our next aim is to show that it is enough to understand those
which are for the form $\langle q \cdot \bh, - \rangle_\bDelta$
for $q \in \bbN$.

We start with the following easy lemma.

\begin{lemma} \label{lemma non-zero}
Let $W \in \add \calR$ be such that $\theta^W (\bd) = 0$ and
$c_\bd^W \neq 0$. If $q \in \bbN$ and $f \in \SI [\bDelta,
\bd]_{\langle q \cdot \bh, - \rangle_\bDelta}$ is non-zero, then
there exists $R \in \calR (\bd)$ such that $c_\bd^W (R) \cdot f
(R) \neq 0$.
\end{lemma}

\begin{proof}
Since $\calR (\bd)$ is an open irreducible subset of $\rep_\bDelta
(\bd)$, the claim follows from Lemmas~\ref{lemma non-zero3}
and~\ref{lemma non-zero2}.
\end{proof}

\begin{proposition} \label{proposition isomorphism}
Let $\br \in \bR$, $\langle \br, \bd \rangle_\bDelta = 0$ and $W
\in \add \calR$ be an ext-minimal representation for $\br -
p^{\br} \cdot \bh$.
\begin{enumerate}

\item \label{point isomorphism1}
If $c_\bd^W = 0$, then $\SI [\bDelta, \bd]_{\langle \br, -
\rangle_\bDelta} = 0$.

\item \label{point isomorphism2}
If $c_\bd^W \neq 0$, then the map
\[
\SI [\bDelta, \bd]_{\langle p^\br \cdot \bh, - \rangle_\bDelta}
\to \SI [\bDelta, \bd]_{\langle \br, - \rangle_\bDelta}, f \mapsto
c_\bd^W \cdot f,
\]
is an isomorphism of vector spaces.

\end{enumerate}
\end{proposition}

\begin{proof}
Let $\Phi : \SI [\bDelta, \bd]_{\langle p^\br \cdot \bh, -
\rangle_\bDelta} \to \SI [\bDelta, \bd]_{\langle \br, -
\rangle_\bDelta}$ be the map given by $\Phi (f) := c_\bd^W \cdot
f$, for $f \in \SI [\bDelta, \bd]_{\langle p^\br \cdot \bh, -
\rangle_\bDelta}$.

It follows from Corollary~\ref{corollary spanning} and
Lemma~\ref{lemma R} that $\SI [\bDelta, \bd]_{\langle \br, -
\rangle_\bDelta}$ is spanned by the functions $c_\bd^V$ for
ext-minimal $V \in \add \calR$ such that $\bdim V = \br$. If $V
\in \add \calR$ is ext-minimal and $\bdim V = \br$, then there
exists an exact sequence $0 \to R \to V \to W \to 0$, where $R \in
\add \calR$ and $\bdim R = p^\br \cdot \bh$. Thus
Corollary~\ref{corollary multsemi}\eqref{point multsemi1} implies
that (up to a non-zero scalar) $c_\bd^V = c_\bd^W \cdot c_\bd^R =
\Phi (c_\bd^R)$. This shows that $\Phi$ is an epimorphism. In
particular, $\SI [\bDelta, \bd]_{\langle \br, - \rangle_\bDelta} =
0$ if $c_\bd^W = 0$. On the other hand, if $c_\bd^W \neq 0$, then
$\Phi$ is a monomorphism (hence an isomorphism) by
Lemma~\ref{lemma non-zero}.
\end{proof}

In the previous papers on the subject the authors have studied
either the semi-invariants on the whole variety $\rep_\bDelta
(\bd)$~\cites{DomokosLenzing2000, DomokosLenzing2002} or on the
closure of $\calR (\bd)$ only~\cite{SkowronskiWeyman1999}.
However, the answers they have obtained did not differ. We have
the following explanation of this phenomena.

\begin{proposition} \label{proposition nonzero}
If $f \in \Bbbk [\rep_\bDelta (\bd)]$ is a non-zero
semi-invariant, then there exists $R \in \calR (\bd)$ such that $f
(R) \neq 0$.
\end{proposition}

\begin{proof}
Fix $\br \in \bR$ such that $f \in \SI [\bDelta, \bd]_{\langle
\br, - \rangle_\bDelta}$. The previous lemma implies that $f =
c_\bd^W \cdot f'$, where $W \in \add \calR$ is an ext-minimal
representation with dimension vector $\br - p^{\br} \cdot \bh$ and
$f' \in \SI [\bDelta, \bd]_{\langle p^{\br} \cdot \bh, -
\rangle_\bDelta}$. Consequently, the claim follows from
Lemma~\ref{lemma non-zero}.
\end{proof}

Observe that this proposition means in particular, that $\SI
[\bDelta, \bd]$ is a domain, hence the product of two non-zero
semi-invariants is non-zero again.

Proposition~\ref{proposition isomorphism} implies that the
subalgebra $\bigoplus_{q \in \bbN} \SI [\bDelta, \bd]_{\langle q
\cdot \bh, - \rangle_\bDelta}$ of $\SI [\bDelta, \bd]$ plays a
crucial role. In Section~\ref{section result} we show that the
study of this subalgebra can be reduced to the case of the
Kronecker quiver. Thus in the next section we recall facts about
semi-invariants for the Kronecker quiver.

\section{The Kronecker quiver} \label{section Kronecker}

Our aim in this section is to collect necessary facts about
representations and semi-invariants for the Kronecker quiver
$K_2$, i.e.\ the quiver $\xymatrix{\vertexU{1} & \vertexU{2}
\ar@/_/[l]_\alpha \ar@/^/[l]^\beta}$ with the empty set of
relations. In this case a sincere separating exact subcategory is
uniquely determined. Let $\calT = \coprod_{\lambda \in
\bbP_\Bbbk^1} \calT_\lambda$ by the sincere separating exact
subcategory of $\ind K_2$.

For $\zeta, \xi \in \Bbbk$ let $N_{\zeta, \xi}$ be the
representation $\xymatrix{\Bbbk & \Bbbk \ar@/_/[l]_\zeta
\ar@/^/[l]^\xi}$. Then the simple objects in $\add \calT$ are
precisely the representations $N_{\zeta, \xi}$ for $(\zeta : \xi)
\in \bbP_k^1$. Moreover, if $(\zeta :\xi), (\zeta' : \xi') \in
\bbP_\Bbbk^1$, then $N_{\zeta, \xi} \simeq N_{\zeta', \xi'}$ if
and only if $(\zeta : \xi) = (\zeta' : \xi')$. Consequently, by
abuse of notation, we will denote $N_{\zeta, \xi}$ by $N_{(\zeta :
\xi)}$ for $(\zeta : \xi) \in \bbP_k^1$. By choosing our
parameterization appropriately we may assume that $N_\lambda \in
\calT_\lambda$ for each $\lambda \in \bbP_k^1$. In particular,
$\tau N_\lambda = N_\lambda$ for each $\lambda \in \bbP_\Bbbk^1$.

The Kronecker quiver can be viewed as the minimal
concealed-canon\-ical bound quiver. Namely, we can embed the
category $\rep K_2$ into the category of representations of an
arbitrary concealed-canonical quiver. We describe a construction
of such an embedding more precisely.

Let $\bDelta$ be a concealed-canonical bound quiver with a sincere
separating exact subcategory $\calR$ of $\ind \bDelta$. Let $R :=
\bigoplus_{\lambda \in \bbP_\Bbbk^1} \bigoplus_{i \in I_\lambda}
R_{\lambda, i}$ for subsets $I_\lambda \subseteq [0, r_\lambda -
1]$ such that $|I_\lambda| = r_\lambda - 1$ (in particular,
$I_\lambda = \varnothing$ if $r_\lambda = 1$), where we use
notation introduced in Section~\ref{section tubular}. Let
$R^\perp$ denote the full subcategory of $\rep \bDelta$, whose
objects are $M \in \rep \bDelta$ such that $\Hom_\bDelta (R, M) =
0 = \Ext_\bDelta^1 (R, M)$. Lenzing and de la
Pe\~na~\cite{LenzingdelaPena1999}*{Proposition~4.2} have proved
that there exists a fully faithful exact functor $F : \rep K_2 \to
\rep \bDelta$ which induces an equivalence between $\rep K_2$ and
$R^\perp$. Moreover, $F$ induces an equivalence between $\calT$
and $R^\perp \cap \calR$. The simple objects in $R^\perp \cap
(\add \calR)$, which are the images of the simple objects in $\add
\calT$, are of the form $R_{\lambda, i_\lambda}^{(r_\lambda)}$ for
$\lambda \in \bbP_\Bbbk^1$, where for $\lambda \in \bbP_\Bbbk^1$
we denote by $i_\lambda$ the unique element of $[0, r_\lambda - 1]
\setminus I_\lambda$. Consequently, (if we choose appropriate
parameterization) $F (N_\lambda) \simeq R_{\lambda,
i_\lambda}^{(r_\lambda)}$ for each $\lambda \in \bbP_\Bbbk^1$.

Let $p \in \bbN$. We define the functions $f^{(0)}_{(p, p)},
\ldots, f^{(p)}_{(p, p)} \in \Bbbk [\rep_{K_2} (p, p)]$ by the
condition: if $V \in \rep_{K_2} (p, p)$, then
\[
\det (S \cdot V_\alpha - T \cdot V_\beta) = \sum_{i \in [0, p]}
S^i \cdot T^{p - i} \cdot f^{(i)}_{(p, p)} (V).
\]
Note that $f^{(0)}_{(p, p)}$, \ldots, $f^{(p)}_{(p, p)}$ are
semi-invariants of weight $(-1, 1)$. If $(\zeta : \xi) \in
\bbP_\Bbbk^1$, then (by choosing a projective presentation of
$N_{\zeta, \xi}$ in an appropriate way) we get
\begin{equation} \label{eq cT}
c_{(p, p)}^{N_{\zeta, \xi}} (V) = \det (\xi \cdot V_\alpha - \zeta
\cdot V_\beta) = \sum_{i \in [0, p]} \xi^i \cdot \zeta^{p - i}
\cdot f^{(i)}_{(p, p)} (V).
\end{equation}
It is well known (see for example~\cite{SkowronskiWeyman2000})
that $\SI [K_2, (p, p)]$ is the polynomial algebra in
$f^{(0)}_{(p, p)}$, \ldots, $f^{(p)}_{(p, p)}$. In particular,
\begin{equation} \label{eq dimension}
\dim_\Bbbk \SI [K_2, (p, p)]_{(-q, q)} = \binom{q + p}{q}
\end{equation}
for each $q \in \bbN$.

We will need the following lemma.

\begin{lemma} \label{lemma equality}
If $f_1, f_2 \in \SI [K_2, (p, p)]_{(-1, 1)}$ and
\[
\{ V \in \rep_{K_2} (p, p) : \text{$f_1 (V) = 0$} \} = \{ V \in
\rep_{K_2} (p, p) : \text{$f_2 (V) = 0$} \},
\]
then \textup{(}up to a non-zero scalar\textup{)} $f_1 = f_2$.
\end{lemma}

\begin{proof}
From the description of $\SI [K_2, (p, p)]$ it follows that $f_1$
and $f_2$ are irreducible, hence the claim follows.
\end{proof}

\section{The main result} \label{section result}

Throughout this section we fix a concealed-canonical bound quiver
$\bDelta$ and a sincere separating exact subcategory $\calR$ of
$\ind \bDelta$. We use freely notation introduced in
Section~\ref{section tubular}. We also fix $\bd \in \bR$ such that
$p := p^\bd > 0$.

First we investigate the algebra $\bigoplus_{q \in \bbN} \SI
[\bDelta, \bd]_{\langle q \cdot \bh, - \rangle_\bDelta}$. We
introduce some notation. For $\lambda \in \bbP_\Bbbk^1$ we denote
by $\calI_\lambda^0$ the set of $i \in [0, r_\lambda - 1]$ such
that $p_{\lambda, i}^\bd = 0$. Observe that $\calI_\lambda^0
\subseteq \calI_\lambda$ for each $\lambda \in \bbP_\Bbbk^1$ (the
sets $\calI_\lambda$ for $\lambda \in \bbP_\Bbbk^1$ were
introduced before Corollary~\ref{corollary generatingbis}). Recall
that, for $\lambda \in \bbP_\Bbbk^1$ and $i \in \calI_\lambda$,
$n_{\lambda, i}$ denotes the minimal $n \in \bbN_+$ such that
$p_{\lambda, i + n}^\bd = 0$. We put $c_\bd^\lambda := \prod_{i
\in \calI_\lambda^0} c_\bd^{R_{\lambda, i}^{(n_{\lambda, i})}}$.
An iterated application of Corollary~\ref{corollary
multsemi}\eqref{point multsemi1} to exact sequences of the
form~\eqref{eq sequence} implies that $c_\bd^\lambda =
c_\bd^{R_{\lambda, i}^{(r_\lambda)}}$ for each $i \in
\calI_\lambda^0$.

We have the following fact.

\begin{lemma} \label{label generating}
The algebra $\bigoplus_{q \in \bbN} \SI [\bDelta, \bd]_{\langle q
\cdot \bh, - \rangle_\bDelta}$ is generated by the semi-invariants
$c_\bd^\lambda$ for $\lambda \in \bbX$.
\end{lemma}

\begin{proof}
This fact has been proved in~\cite{Bobinski2011}, but for
completeness we include its (shorter) proof here.

Fix $q \in \bbN$. Proposition~\ref{corollary spanning} and
Lemma~\ref{lemma R} imply that $\SI [\bDelta, \bd]_{\langle q
\cdot \bh, - \rangle_\bDelta}$ is spanned by the semi-invariants
$c_\bd^V$ for ext-minimal $V \in \add \calR$ with dimension vector
$q \cdot \bh$. Fix such $V$. Since $V$ is ext-minimal with
dimension vector $q \cdot \bh$, $V = \bigoplus_{\lambda \in \bbX}
R_{\lambda, i_\lambda}^{(k_\lambda \cdot r_\lambda)}$, where $\bbX
\subseteq \bbP_\Bbbk^1$ and $i_\lambda \in [0, r_\lambda - 1]$ and
$k_\lambda \in \bbN_+$ for each $\lambda \in \bbX$. Moreover,
Lemma~\ref{lemma indecomposable} implies that $i_\lambda \in
\calI_\lambda^0$ for each $\lambda \in \bbX$. An iterated
application of Corollary~\ref{corollary multsemi}\eqref{point
multsemi1} to exact sequences of the form~\eqref{eq sequence}
implies that $c_\bd^{R_{\lambda, i_\lambda}^{(k_\lambda \cdot
r_\lambda)}} = (c_\bd^\lambda)^{k_\lambda}$ for each $\lambda \in
\bbX$. Consequently, $c_\bd^V = \prod_{\lambda \in \bbX}
(c_\bd^\lambda)^{k_\lambda}$ by Lemma~\ref{lemma directsum}, hence
the claim follows.
\end{proof}

The following fact is crucial.

\begin{proposition} \label{proposition regular}
There exists a regular map
\[
\Phi : \rep_{K_2} (p, p) \to \rep_\bDelta (\bd)
\]
such that $\Phi^*$ induces an isomorphism
\[
\bigoplus_{q \in \bbN} \SI [\bDelta, \bd]_{\langle q \cdot \bh, -
\rangle_\bDelta} \to \bigoplus_{q \in \bbN} \SI [K_2, (p,
p)]_{(-q, q)}
\]
of $\bbN$-graded rings and \textup{(}up to a non-zero
scalar\textup{)} $\Phi^* (c_\bd^\lambda) = c_{(p, p)}^{N_\lambda}$
for each $\lambda \in \bbP_\Bbbk^1$.
\end{proposition}

\begin{proof}
For each $\lambda \in \bbP_\Bbbk^1$ we fix $i_\lambda \in
\calI_\lambda^0$. From Section~\ref{section Kronecker} we know
that there exists a fully faithful exact functor $F : \rep K_2 \to
\rep \bDelta$ such that $F (N_\lambda) \simeq R_{\lambda,
i_\lambda}^{(r_\lambda)}$ for each $\lambda \in \bbP_\Bbbk^1$.
Observe that for each $R \in \add (\coprod_{\lambda \in
\bbP_\Bbbk^1 \setminus \bbX_0} \calR_\lambda)$ (recall that
$\bbX_0$ is the set of all $\lambda \in \bbP_\Bbbk^1$ such that
$r_\lambda > 1$) there exists $N \in \calT$ with $F (N) \simeq R$.

Put $E_1 := F (S_1)$ and $E_2 := F (S_2)$, where $S_i$ is the
simple representation of $K_2$ at $i$, for $i \in \{ 1, 2 \}$,
i.e.\
\[
S_1 := \xymatrix{\Bbbk & 0 \ar@/_/[l] \ar@/^/[l]} \qquad
\text{and} \qquad S_2 := \xymatrix{0 & \Bbbk \ar@/_/[l]
\ar@/^/[l]}.
\]
Then~\cite{RiedtmannSchofield1990}*{Proposition~2.3} (see
also~\cite{Chindris2010}*{Proposition~5.2}) implies that there
exists a regular map $\Phi' : \rep_{K_2} (p, p) \to \rep_\bDelta
(p \cdot \bh)$ such that $\Phi' (N) \simeq F (N)$ for each $N \in
\rep_{K_2}  (p, p)$. Moreover, there exists a morphism $\varphi :
\GL (p, p) \to \GL (p \cdot \bh)$ of algebraic groups such that
$\Phi' (g \ast N) = \varphi (g) \ast \Phi' (N)$, for all $g \in
\GL (p, p)$ and $N \in \rep_\bDelta (p \cdot \bh)$, and
\begin{equation} \label{eq theta}
\chi^\theta (\varphi (g)) = (\det (g (1))^{\theta (\bdim E_1)}
\cdot (\det g (2))^{\theta (\bdim E_2)},
\end{equation}
for all $g \in \GL (p, p)$ and $\theta \in \bbZ^{\Delta_0}$.

Let $W \in \add \calR$ be an ext-minimal representation for $\bd'
:= \bd - p \cdot \bh$. We define $\Phi : \rep_{K_2} (p, p) \to
\rep_\bDelta (\bd)$ by $\Phi (N) := \Phi' (N) \oplus W$ for $N \in
\rep_{K_2} (p, p)$.

Let $q \in \bbN$. We show that $\Phi^* (f)$ is a semi-invariant of
weight $(-q, q)$ for each $f \in \SI [\bDelta, \bd]_{\langle q
\cdot \bh, - \rangle_\bDelta}$. Using Proposition~\ref{proposition
Domokos} and Lemma~\ref{lemma R} it suffices to show that $\Phi^*
(c^V)$ is a semi-invariant of weight $(-q, q)$ for each
representation $V$ of $\bDelta$ with dimension vector $q \cdot
\bh$. Now, if $g \in \GL (p, p)$ and $N \in \rep_{K_2} (p, p)$,
then
\begin{align*}
(\Phi^* (c_\bd^V)) (g^{-1} \ast N) & = c_\bd^V (W \oplus \Phi'
(g^{-1} \ast N))
\\
& = c_{\bd'}^V (W) \cdot c_{p \cdot \bh}^V (\varphi (g^{-1}) \ast
\Phi' (N))
\\
& = c_{\bd'}^V (W) \cdot \chi^{\langle q \cdot \bh, -
\rangle_\bDelta} (\varphi (g)) \cdot c_{p \cdot \bh}^V (\Phi (N))
\\
& = \chi^{\langle q \cdot \bh, - \rangle_\bDelta} (\varphi (g))
\cdot (\Phi^* (c_\bd^V)) (N),
\end{align*}
where the second and the last equalities follow from
Lemma~\ref{lemma directsumbis}. Using~\eqref{eq theta} we get
\[
\chi^{\langle q \cdot \bh, - \rangle_\bDelta} (\varphi (g)) =
(\det (g (1))^{-q} \cdot (\det (g (2)))^q,
\]
since
\[
\langle \bh, \bdim E_i \rangle_\bDelta = \langle (1, 1), \bdim S_i
\rangle_{K_2} = (-1)^i
\]
for each $i \in \{ 1, 2 \}$ (we use here that $F$ is exact).

The above implies that $\Phi^*$ induces a homomorphism
\begin{equation} \label{eq iso}
\bigoplus_{q \in \bbN} \SI [\bDelta, \bd]_{\langle q \cdot \bh, -
\rangle_\bDelta} \to \bigoplus_{q \in \bbN} \SI [K_2, (p,
p)]_{(-q, q)}
\end{equation}
of $\bbN$-graded rings. We need to show that this is an
isomorphism.

First we show $\Phi^* (f) \neq 0$ for each non-zero semi-invariant
$f$ (in particular, this will imply that~\eqref{eq iso} is a
monomorphism). Let
\begin{multline*}
\calZ := \{ M \in \rep_\bDelta (\bd) : \text{there exists}
\\
\text{$N \in \rep_{K_2} (p, p)$ such that $M \simeq W \oplus \Phi
(N)$} \}.
\end{multline*}
In other words, $\calZ$ in the closure of the image of $\Phi$
under the action of $\GL (\bd)$. Using
Proposition~\ref{proposition nonzero} it suffices to show that
$\calZ$ contains a non-empty open subset of $\calR (\bd)$. Let
\begin{multline*}
\calU := \{ M \in \calR (\bd) : \text{$c_\bd^\lambda (M) \neq 0$
for each $\lambda \in \bbX_0$}
\\
\text{and $\dim_\Bbbk \End_\bDelta (M) = p + \langle \bd, \bd
\rangle_\bDelta$} \}.
\end{multline*}
Since the function
\[
\rep_\bDelta (\bd) \ni M \mapsto \dim_\Bbbk \End_\bDelta (M) \in
\bbZ
\]
is upper semi-continuous, \eqref{eq minimum} implies that $\calU$
is a non-empty open subset $\calR (\bd)$, which consists of
ext-minimal representations. In particular, if $M \in \calU$, then
there exists an exact sequence of the form $0 \to R \to M \to W
\to 0$ with $R \in \add \calR$ such that $\bdim R = p \cdot \bh$.
If $p_\lambda^R \neq 0$ for some $\lambda \in \bbX_0$, then
$\Hom_\bDelta (R_{\lambda, i_\lambda}^{(r_\lambda)}, R) \neq 0$.
Consequently, $\Hom_\bDelta (R_{\lambda, i_\lambda}^{(r_\lambda)},
M) \neq 0$, hence $c_\bd^\lambda (M) = 0$, contradiction. Thus
$p_\lambda^R = 0$ for each $\lambda \in \bbX_0$, hence $M \simeq W
\oplus R$ and $R \in \add (\coprod_{\lambda \in \bbP_\Bbbk^1
\setminus \bbX_0} \calR_\lambda)$. In particular, there exists $N
\in \rep \calT$ such that $F (N) \simeq R$, hence $M \in \calZ$.

Now we fix $\lambda  \in \bbP_k^1$. We show that \textup{(}up to a
non-zero scalar\textup{)} $\Phi^* (c_\bd^\lambda) = c_{(p,
p)}^{N_\lambda}$ for each $\lambda \in \bbP_\Bbbk^1$. According to
Lemma~\ref{label generating} this will imply that~\eqref{eq iso}
is an epimorphism, hence finish the proof. Fix $N \in \rep_{K_2}
(p, p)$. Then
\[
(\Phi^* (c_\bd^\lambda)) (N) = 0 \text{ if and only if }
\Hom_\bDelta (R_{\lambda, i_\lambda}^{(r_\lambda)}, F (N)) \neq 0.
\]
Since $R_{\lambda, i_\lambda}^{(r_\lambda)} \simeq F (N_\lambda)$
and $F$ is fully faithful,
\[
(\Phi^* (c_\bd^\lambda)) (N) = 0 \text{ if and only if }
\Hom_{K_2} (N_\lambda, N) \neq 0.
\]
Similarly, if $N \in \rep_{K_2} (p, p)$, then
\[
c_{(p, p)}^{N_\lambda} (N) = 0 \text{ if and only if } \Hom_{K_2}
(N_\lambda, N) \neq 0.
\]
Consequently, the claim follows from Lemma~\ref{lemma equality}.
\end{proof}

\begin{corollary} \label{corollary dimension}
If $\br \in \bR$ and $\SI [\bDelta, \bd]_{\langle \br, -
\rangle_\bDelta} \neq 0$, then
\[
\dim_\Bbbk \SI [\bDelta, \bd]_{\langle \br, - \rangle_\bDelta} =
\binom{p^\br + p}{p^\br}.
\]
\end{corollary}

\begin{proof}
Proposition~\ref{proposition isomorphism}\eqref{point
isomorphism2} implies that
\[
\dim_\Bbbk \SI [\bDelta, \bd]_{\langle \br, - \rangle_\bDelta} =
\dim_\Bbbk \SI [\bDelta, \bd]_{\langle p^{\br} \cdot \bh, -
\rangle_\bDelta}.
\]
Next,
\[
\dim_\Bbbk \SI [\bDelta, \bd]_{\langle p^\br \cdot \bh, -
\rangle_\bDelta} = \dim_\Bbbk \SI [K_2, (p, p)]_{(-p^\br, p^\br)}
\]
by Proposition~\ref{proposition regular}, hence the claim follows
from~\eqref{eq dimension}.
\end{proof}

Let $\Phi : \rep_{K_2} (p, p) \to \rep_\bDelta (\bd)$ be a regular
map constructed in Proposition~\ref{proposition regular}. For $j
\in [0, p]$ we denote by $f_\bd^{(j)}$ the inverse image of
$f_{(p, p)}^{(j)}$ under $\Phi^*$. Then \eqref{eq cT} implies that
(up to a non-zero scalar)
\begin{equation} \label{eq cTbis}
c_\bd^{(\zeta : \xi)} = \sum_{j \in [0, p]} \xi^j \cdot \zeta^{p -
j} \cdot f^{(j)}_\bd
\end{equation}
for each $(\zeta : \xi) \in \bbP_\Bbbk^1$. As the first
application we get the following (smaller) set of generators of
$\SI [\bDelta, \bd]$.

\begin{proposition} \label{proposition generators}
The algebra $\SI [\bDelta, \bd]$ is generated by the
semi-in\-var\-iants $f_\bd^{(0)}$, \ldots, $f_\bd^{(p)}$ and
$c_\bd^{R_{\lambda, i + 1}^{(n_{\lambda, i})}}$ for $\lambda \in
\bbX_0$ and $i \in \calI_\lambda$.
\end{proposition}

\begin{proof}
Recall from Corollary~\ref{corollary generatingbis} that the
algebra $\SI [\bDelta, \bd]$ is generated by the semi-invariants
$c_\bd^{R_{\lambda, i + 1}^{(n_{\lambda, i})}}$ for $\lambda \in
\bbP_\Bbbk^1$ and $i \in \calI_\lambda$. Thus we only need to
express, for each $\lambda \in \bbP_\Bbbk^1 \setminus \bbX_0$ and
$i \in \calI_\lambda$, $c_\bd^{R_{\lambda, i + 1}^{(n_{\lambda,
i})}}$ as the polynomial in the semi-invariants listed in the
proposition. However, if $\lambda \in \bbP_\Bbbk^1 \setminus
\bbX_0$ and $i \in \calI_\lambda$, then $c_\bd^{R_{\lambda, i +
1}^{(n_{\lambda, i})}} = c_\bd^\lambda$, hence the claim follows
from~\eqref{eq cTbis}.
\end{proof}

We give another formulation of Proposition~\ref{proposition
generators}. Let $\calA$ be the polynomial algebra in the
indeterminates $S_0$, \ldots, $S_p$ and $T_{\lambda, i}$ for
$\lambda \in \bbX_0$ and $i \in \calI_\lambda$.
Proposition~\ref{proposition generators} says that the
homomorphism $\Psi : \calA \to \SI [\bDelta, \bd]$ given by the
formulas: $\Psi (S_j) := f_\bd^{(j)}$, for $j \in [0, p]$, and
$\Psi (T_{\lambda, i}) := c_\bd^{R_{\lambda, i + 1}^{(n_{\lambda,
i})}}$, for $\lambda \in \bbX_0$ and $i \in \calI_\lambda$, is an
epimorphism. Our last aim is to describe its kernel.

First, we introduce an $\bR$-grading in $\calA$ by specifying the
degrees of the indeterminates as follows: $\deg (S_j) := \bh$ for
$j \in [0, p]$ and $\deg (T_{\lambda, i}) := \be_{\lambda,
i}^{n_{\lambda, i}}$, for $\lambda \in \bbX_0$ and $i \in
\calI_\lambda$. Note that $\Psi$ is a homogenenous map, i.e.\
$\Psi (\calA_\br) = \SI [\bDelta, \bd]_{\langle \br, -
\rangle_\bDelta}$ for each $\br \in \bR$.

Let $\bR_0$ be the submonoid of $\bR$ generated by the elements
$\bh$ and $\be_{\lambda, i}^{n_{\lambda, i}}$ for $\lambda \in
\bbX_0$ and $i \in \calI_\lambda$. Obviously, if $\br \in \bR$,
then $\calA_\br \neq 0$ if and only if $\br \in \bR_0$. Similarly,
Corollary~\ref{corollary generatingbis} implies that $\SI
[\bDelta, \bd]_{\langle \br, - \rangle_\bDelta} \neq 0$ if and
only if $\br \in \bR_0$ (recall that $\SI [\bDelta, \bd]$ is a
domain).

\begin{lemma} \label{lemma dimensionA}
If $\br \in \bR_0$, then
\[
\dim_\Bbbk \calA_\br = \binom{p^\br + p + |\bbX_0|}{p^\br}.
\]
\end{lemma}

\begin{proof}
One easily observes that there is an isomorphism $\calA_{p^\br
\cdot \bh} \to \calA_\br$ of vector spaces (induced by multiplying
by the unique monomial of degree $\br - p^\br \cdot \bh$).
Moreover, $\bigoplus_{q \in \bbN} \calA_{q \cdot \bh}$ is the
polynomial algebra generated by $S_0$, \ldots, $S_p$ and $\prod_{i
\in \calI_\lambda^0} T_{\lambda, i}$ for $\lambda \in \bbX_0$. Now
the claim follows.
\end{proof}

The formula~\eqref{eq cTbis} implies that for each $\lambda \in
\bbX_0$ there exist $\zeta_\lambda, \xi_\lambda \in \Bbbk$ such
that
\[
\prod_{i \in \calI_\lambda^0} c_\bd^{R_{\lambda, i +
1}^{(n_{\lambda, i})}} = \sum_{j \in [0, p]} \xi_\lambda^j \cdot
\zeta_\lambda^{p - j} \cdot f^{(j)}_\bd.
\]
Obviously, $(\zeta_\lambda, \xi_\lambda) \neq (0, 0)$ and
$(\zeta_\lambda : \xi_\lambda) = \lambda$.

\begin{proposition}
We have
\[
\Ker \Psi = \Bigl( \sum_{j \in [0, p]} \xi_\lambda^j \cdot
\zeta_\lambda^{p - j} \cdot S_j - \prod_{i \in \calI_\lambda^0}
T_{i, \lambda} : \lambda \in \bbX_0 \Bigr).
\]
\end{proposition}

\begin{proof}
Let
\[
\calJ := \Bigl( \sum_{j \in [0, p]} \xi_\lambda^j \cdot
\zeta_\lambda^{p - j} \cdot S_j - \prod_{i \in \calI_\lambda^0}
T_{i, \lambda} : \lambda \in \bbX_0 \Bigr).
\]
Obviously, $\Ker \Psi \subseteq I$. Observe that both $\Ker \Psi$
and $\calJ$ are graded ideals (with respect to the grading
introduced above). Consequently, in order to prove our claim it
suffices to show that $\dim_\Bbbk \calJ_\br = \dim_\Bbbk \Ker
\Psi_\br$ for each $\br \in \bR_0$.

We already know from Lemma~\ref{lemma dimensionA} and
Corollary~\ref{corollary dimension} that
\begin{align*}
\dim_\Bbbk \Ker \Psi_\br & = \dim_\Bbbk A_\br - \dim_\Bbbk \SI
[\bDelta, \br]_{\langle \br, - \rangle_\bDelta}
\\
& = \binom{p^\br + p + |\bbX_0|}{p^\br} - \binom{p^\br + p}{p^\br}
\end{align*}
for each $\br \in \bR_0$. On the other hand, similarly as in the
proof of Lemma~\ref{lemma dimensionA}, we show that $\dim_\Bbbk
\calJ_\br = \dim_\Bbbk \calJ_{p^\br \cdot \bh}$ for each $\br \in
\bR_0$. Moreover, the algebra $\bigoplus_{q \in \bbN} (\calA /
\calJ)_{q \cdot \bh}$ is obviously the polynomial algebra in
$p^\br + p$ indeterminates. This, together with Lemma~\ref{lemma
dimensionA}, immediately implies our claim.
\end{proof}

We may summarize our considerations in the following theorem
(compare~\cite{SkowronskiWeyman1999}*{Theorem~1.1}).

\begin{theorem}
We have the isomorphism
\[
\SI [\bDelta, \bd] \simeq \calA / \Bigl( \sum_{j \in [0, p]}
\xi_\lambda^j \cdot \zeta_\lambda^{p - j} \cdot S_j - \prod_{i \in
\calI_\lambda^0} T_{i, \lambda} : \lambda \in \bbX_0 \Bigr).
\]
If
\[
i (\bd) := \{ \lambda \in \bbX_0 : \text{$|\calI_\lambda| > 1$}
\},
\]
then $\SI [\bDelta, \bd]$ is a complete intersection given by
$\max (0, i (\bd) - p - 1)$ equations. In particular, $\SI
[\bDelta, \bd]$ is  polynomial algebra if and only if $i (\bd)
\leq p + 1$.
\end{theorem}

\bibsection

\end{document}